\documentclass[reqno]{amsart}
\usepackage[hidelinks]{hyperref}
\usepackage{amsmath, amssymb, amsthm, epsfig}
\usepackage{hyperref, latexsym}
\usepackage{url}
\usepackage[mathscr]{euscript}
\usepackage{mathabx}
\usepackage{color}
\usepackage{fullpage} 
\usepackage{setspace}
\usepackage{todonotes}
\usepackage[final]{changes}
\newtheorem{lemma}{Lemma}[section]

\def\today{\ifcase\month\or
	January\or February\or March\or April\or May\or June\or
	July\or August\or September\or October\or November\or December\fi
	\space\number\day, \number\year}

\newtheorem{theorem}{Theorem}

\theoremstyle{definition}

\theoremstyle{remark}

\newcommand{\R}{\mathbb{R}}
\newcommand{\N}{\mathbb{N}}

\newcommand{\Z}{\mathbb{Z}}

\newcommand{\hh}{\tfrac12}
\newcommand{\ds}{\text{\rm d}s}

\newcommand{\dt}{\text{\rm d}t}
\renewcommand{\d}{\text{\rm d}}
\newcommand{\du}{\text{\rm d}u}

\newcommand{\dx}{\text{\rm d}x}

\newcommand{\dy}{\text{\rm d}y}

\newcommand{\dnu}{\text{\rm d}\nu}

\newcommand{\re}{{\rm Re}\,}

\begin{document}
	\title[On Goldbach numbers in short intervals]{On Goldbach numbers in short intervals}
	\author[Chirre and Hagen]{Andr\'{e}s Chirre and Markus Valås Hagen}
	\address{Departamento de Ciencias - Sección Matemáticas, Pontificia Universidad Católica del Perú, Av. Universitaria 1801, San Miguel 15088, Lima, Perú}
	\email{cchirre@pucp.edu.pe}
	\address{Department of Mathematical Sciences, Norwegian University of Science and Technology (NTNU), 7491 Trondheim, Norway} 
	\email{markus.v.hagen@ntnu.no}
	
	\thanks{MVH is supported in part by Grant 334466 of the Research Council of Norway.}
	
	\allowdisplaybreaks
	\numberwithin{equation}{section}
	
	\maketitle  
	
	\begin{abstract}
		Assuming the Riemann \replaced{H}{h}ypothesis, we prove that for all $x\geq 2$, there exists at least one even integer within the interval $(x, x+123\log^2x]$, that can be expressed as the sum of two primes. This result is an improvement over the recent work of Cully-Hugill and Dudek, who obtained the constant 
		$9696$ instead of $123$.
	\end{abstract}
	%\tableofcontents

	\vspace{0.5cm}
	
	%\markus{New title suggestion: On Goldbach numbers in short explicit intervals.}
	
	\section{Introduction}
	One of the most famous open problems in number theory is the Goldbach conjecture. It states that every even number greater than $2$ can be written as a sum of two primes. Although the conjecture still remains open, a significant amount of progress has been made towards it. One of them is the big breakthrough achieved by Helfgott \cite{HaraldG}, \added{building on work of Vinogradov}, who proved the weak Goldbach conjecture - the statement that every odd number greater than $5$ can be written as a sum of three primes.
	
	%Montgomery--Vaughan \cite{MV2} showed in the 1970's that the density of the numbers where Goldbach possibly fails, is $0$. This was already known due to Hardy--Littlewood \cite{HardyLittlewood}, on the assumption of the Riemann Hypothesis (RH). 

	Since the original Goldbach conjecture still seems to be out of reach, it is natural to ask on how short intervals we can exhibit a number that is the sum of two primes. To make this precise, \replaced{the notion of a Goldbach number has been introduced}{we introduce the notion of a Goldbach number} - this is a positive integer that is the sum of two odd primes. The Goldbach conjecture is the statement that there is a Goldbach number in any interval $(x,x+2]$ for any $x> 4$. If we consider the interval $(x,x+H]$, how small can we take $H$ and still ensure the existence of a Goldbach number in our interval?
	
	Using strong unconditional results on zeroes of $L$-functions due to Gallagher \cite{Gallagher}, Montgomery--Vaughan \cite{MV2} were able to prove that there is a Goldbach number in the interval $(x,x+x^{\frac{7}{72}+\varepsilon})$. Assuming the Riemann \replaced{H}{h}ypothesis (RH), the first result on Goldbach numbers in short intervals was obtained by Linnik in \cite{Linnik}, where he proved that $[x,x+\log^{3+\varepsilon}x]$ contains a Goldbach number for sufficiently large $x$. Later, Katai \cite{KATAI} and Montgomery--Vaughan \cite{MV2}, improved the interval independently to $[x,x+H]$ where $H=O(\log^{2}x)$, and with $x$ still sufficiently large. \added{The order of $H$ hasn't been improved since.} %For both these two works, the result follows rather easily from an estimate on primes in short intervals due to Selberg \cite{S}. Results of the same strength, but that are more close to the approach of Linnik, appears in \cite{GOLD} and \cite{PERELLI}. 
	
	In a recent paper \cite{Dudek}, Cully-Hugill and Dudek made the \replaced{method}{proof} of Montgomery and Vaughan explicit. This allowed them to prove that there is a Goldbach number in the interval $(x, x+9696\log^2x]$ for all $x\geq 2$, assuming RH. The heart of their proof lies in finding an explicit bound for an integral, first studied by Selberg in  \cite{S}. For $x\geq 1$ and $\delta>0$, it is defined by
	\begin{align} \label{8_19pm}
		J_\theta(x,\delta)=\int_{1}^{x}\big(\theta((1+\delta)y)-\theta(y)-\delta y\big)^2\dy,
	\end{align}
	where $\theta(x)=\sum_{p\leq x}\log p$. Assuming RH, Cully-Hugill--Dudek proved that $J_\theta(x,\delta)<202\delta x^2\log^2x$, for $x\geq 10^8$ and $\delta\in (0,10^{-8}]$. Their method focuses on bounding the second moment of the logarithmic derivative of the Riemann zeta-function in the critical strip, which they control explicitly via Selberg's moment formula.

	The purpose of this paper is to improve on the interval $(x,x+9696\log^2x]$. We \replaced{establish}{prove} the following result.
	\begin{theorem} \label{6_35pm2}
		Assume \replaced{RH}{the Riemann hypothesis}. \added{Then there} is a Goldbach number in the interval $(x, x+123\log^2x]$ for all $x\geq 2$.
	\end{theorem}
	
	\replaced{Ultimately, we will deduce Theorem \ref{6_35pm2} from an explicit bound for $J_\theta(x,\delta)$, like Cully-Hugill--Dudek. However, our approach for bounding this quantity differs significantly from theirs.}{Our approach also requires bounding \( J_\theta(x,\delta) \), but it differs from that of Cully-Hugill--Dudek.} We work directly with the zeros of the Riemann zeta-function, employing the explicit formula for the Chebyshev function $
	\psi(x)=\sum_{n\leq x}\Lambda(n), 
	$
	where $\Lambda(n)$ is the von Mangoldt function defined by $\Lambda(n)=\log p$ if $n=p^k$ for a prime $p$ and $k\in \N$, and $\Lambda(n)=0$ otherwise. Then, we use an averaging technique introduced by Saffari and Vaughan in \cite{SV}, to bound\footnote{In fact, in \cite{SV}, it is proved that $J_\psi(x,\delta)\ll \delta x^2 \log^2({2}/{\delta})$ for $x\geq 4$ and $0<\delta\leq 1$.} explicitly
	\begin{align}  \label{10_48pm}
		J_\psi(x,\delta): 
		=\int_{1}^{x}\big(\psi((1+\delta)y)-\psi(y)-\delta y\big)^2\dy.
	\end{align}
	The rest of the proof is then devoted to pass from $J_\psi(x,\delta)$ to $J_\theta(x,\delta)$ which is a surprisingly more delicate process than one would anticipate. This \replaced{seems to be}{is} mainly because we are working with primes in short intervals.
	
	The structure of the paper is as follows. In Section~\ref{10_38pm}, we present several explicit estimates for some objects related to the Riemann zeta-function and its non-trivial zeros. Section~\ref{10_39pm} is devoted to establishing an explicit bound for $J_\psi(x,\delta)$ through the averaging method. Building upon this result, in Section~\ref{10_39pm22} we derive a bound for $J_\theta(x,\delta)$ by carefully analyzing the associated error term. In Section~\ref{10_39pm2}, we prove Theorem~\ref{6_35pm2} as a consequence of the explicit estimate for $J_\theta(x,\delta)$. Finally, in Section~\ref{6_46AM} we write \replaced{an appendix}{and Appendix} with some explicit bounds for certain sums of prime numbers that appears in our proof.

	\section*{Acknowledgements}
	We would like to thank Adrian Dudek, Winston Heap, Harald Helfgott, Kristian Seip, Timothy Trudgian and Denis Zelent %with help in veriyfing some of the numerical parts, appearing in the appendix. 
	for their valuable suggestions on earlier versions of the paper. Some of the work present in this paper was carried out while MVH was visiting Louis-Pierre Arguin at the University of Oxford in Spring 2025. He would like to thank the university for the hospitality and the excellent working conditions provided there.

	%The following lemma is somewhat technical, but it will help simplify our computations.
	%\begin{lemma}
	%	Let $f,g:[a,b]\to\C$ two integrable functions. Then, for any $\eta>0$ we have
	%	$$
	%	\int_{a}^{b}|f(x)+g(x)|^2\dx\leq (1+\eta)	\int_{a}^{b}|f(x)|^2\dx +  \left(1+\dfrac{1}{\eta}\right)\int_{a}^{b}|g(x)|^2\dx.
	%	$$
	%\end{lemma}
	%\begin{proof}
	%	This result is a straightforward consequence of the MA–MG inequality.
	%\end{proof}

	\vspace{0.5cm}
	%The paper is organized as follows. Some explicit estimates related to the non-trivial zeros of the Riemann zeta-function are given in Section \ref{10_38pm}. In Section \ref{10_39pm} we establish the explicit bound for $J_\psi(x,\delta)$ using the averaging technique. In Section \ref{10_39pm22} we obtain the bound for $J_\theta(x,\delta)$, using the bound for $J_\psi(x,\delta)$ and bounding in a deep way the error term. Finally, in Section \ref{10_39pm2} we deduce Theorem \ref{6_35pm2} from the explicit bound for $J_\theta(x,\delta)$.
	%	(see \cite[Lemma 3.11, p. 67]{HaraldG})
	%	\begin{align*} %\label{3_53pm2}
		%		\frac{\Gamma'}{\Gamma}(z)= \log z -\frac{1}{2z}+ O^*\left(\dfrac{1}{4|z|^2}\right)\!, \,\,\,\,\,\mbox{for}\,\,\,\, \re{z}\geq 0,
		%	\end{align*}
	%	it follows that for $\alpha\geq 0$ and $t\neq 0$
	%	\begin{align*}
		%		\re\dfrac{\Gamma'}{\Gamma}\bigg(\dfrac{\alpha%}{2}+1+\dfrac{it}{2}\bigg)\leq \log |t| - \log 2 + \frac{(\alpha+2)^2}{2t^2}  - \frac{\alpha+1}{(\alpha+2)^2+t^2}.
		%	\end{align*}

	\section{Lemmas related to the Riemann zeta-function and its zeros} \label{10_38pm}
	
	Let $\zeta(s)$ be the Riemann zeta-function, and assume \replaced{RH}{Riemann hypothesis}, i.e. all non-trivial zeros of $\zeta(s)$ have the form $\rho=\hh +i\gamma$ where $\gamma\in \R$.
	
	%\begin{lemma} \label{4_23pm} Let $N(t)$ the number of non-trivial zeros of $\zeta(s)$ with imaginary part $0<\gamma\leq T$. Then
	%	$$
	%	N(t+1)-N(t)\leq \log t,
	%	$$
	%	for all $t\geq e$.
	%\end{lemma} 
	
	\begin{lemma} \label{7_23AM} Assume \replaced{RH}{the Riemann hypothesis}. Then, for $1\leq \sigma\leq 2$ and $|t|\geq 100$ we have\footnote{\added{We should mention that one can do better - }\replaced{t}{T}he classical conditional bound for the logarithmic derivative of the Riemann zeta-function at the point $1+it$ is $O(\log \log t)$. The best explicit result is given by the authors and Simoni\v{c} in \cite[Theorem 5]{ChiVaSI}, \replaced{but this is only useful}{which is useful} when $t$ is really large.} 
		$$
		\left|\dfrac{\zeta'}{\zeta}(\sigma+it)\right|\leq  4\log |t|.
		$$
	\end{lemma}
	\begin{proof}
		We adopt \deleted{the} Backlund's approach \replaced{from}{in} \cite{Backlund}, \replaced{where he}{who} established bounds for $|\zeta(\sigma + it)|$. Let $N\geq 3$, $1<\sigma<2$ and $|t|\geq 100$. By the Dirichlet series representation of $\zeta'/\zeta$, and integration by parts, one can see that
		\begin{align} \label{12_16_29_04}
			-\dfrac{\zeta'}{\zeta}(\sigma+it)=\sum_{n=1}^{N-1}\dfrac{\Lambda(n)}{n^{\sigma+it}} - \dfrac{\psi(N^{-})-N}{N^{\sigma+it}} + \dfrac{N^{1-({\sigma+it})}}{{\sigma+it}-1} + ({\sigma+it})\int_{N}^\infty\dfrac{\psi(y)-y}{y^{{\sigma+it}+1}}\dy.
		\end{align}	
		Let us bound each term on the right-hand side of \eqref{12_16_29_04}. By \cite[Lemma 10]{JRT} we get
		$$\left|\sum_{n=1}^{N-1}\dfrac{\Lambda(n)}{n^{\sigma+it}}\right|\leq \sum_{n=1}^{N-1}\dfrac{\Lambda(n)}{n}\leq \log (N-1)-\gamma + \dfrac{1.3}{\log^2(N-1)}.$$
		
		\replaced{We now proceed to bound the other terms appearing in \eqref{12_16_29_04}. First we observe the trivial bound}{On the other hand, clearly} 
		$
		|{N^{1-({\sigma+it})}}/{({\sigma+it}-1)}|\leq 1/|t|.
		$
		To bound the \replaced{two}{other} terms involving $\psi$, we recall an explicit \replaced{conditional bound}{version} of the \replaced{error term in the prime number theorem}{prime number theorem error term} \deleted{under the Riemann hypothesis} (see \cite[Theorem 10]{Sch}): for all $y\ge 73.2$ we have
		$|\psi(y)-y|\leq{\sqrt{y}\log^2y}/{8\pi}$. \replaced{Thus, if $N\neq p^k$ and $N\geq 74$, we have}{Now, suppose that $N\neq p^k$ and $N\geq 74$. Thus}
		$$
		\left| \dfrac{\psi(N^{-})-N}{N^{\sigma+it}}\right|\leq \dfrac{\log^2N}{8\pi\sqrt{N}}.
		$$
		We bound the last term as follows
		$$
		\left|({\sigma+it})\int_{N}^\infty\dfrac{\psi(y)-y}{y^{{\sigma+it}+1}}\dy\right|\leq \dfrac{|{\sigma+it}|}{8\pi}\int_{N}^\infty\dfrac{\sqrt{y}\log^2y}{y^{2}}dy = \dfrac{|{\sigma+it}|}{4\pi}\left(\dfrac{\log^2 N+4\log N + 8}{\sqrt{N}}\right).
		$$
		%\markus{Here you say $t\geq 10^4$, although later $t\geq 100$ only. Shouldn't it be $t\geq 100$ all the way through? This means that it should say $1.01|t|$ in the bound instead of $1.001|t|$. However, the bound $4\log|t|$ is still valid.}
		So, since $t\geq 100$ and $1<\sigma<2$, we obtain
		$$
		\left|{(\sigma+it)}\int_{N}^\infty\dfrac{\psi(y)-y}{y^{{\sigma+it}+1}}\dy\right|\leq  \dfrac{1.001|t|}{4\pi}\left(\dfrac{\log^2 N+4\log N + 8}{\sqrt{N}}\right).
		$$
		Thus, in \eqref{12_16_29_04},
		$$
		\left|\dfrac{\zeta'}{\zeta}(\sigma+it)\right|\leq  \log (N-1)+\dfrac{1.001|t|}{4\pi}\left(\dfrac{\log^2 N+4\log N + 8}{\sqrt{N}}\right)-\gamma + \dfrac{1.3}{\log^2(N-1)} + \dfrac{1}{|t|}+\dfrac{\log^2N}{8\pi\sqrt{N}}.
		$$
		\replaced{Choosing}{Thus, choosing} $N=[|t|^{4}]+1$ we get
		$$
		\left|\dfrac{\zeta'}{\zeta}(\sigma+it)\right|\leq  4\log |t|.
		$$
		Finally, by continuity we arrive at the desired result for $1\leq\sigma\leq 2$ and $|t|\geq 100$.
	\end{proof}	
	
	\begin{lemma} \label{12_30am} Assume \replaced{RH}{the Riemann hypothesis}. Then, the following is true, where 	the sums run over the imaginary parts $\gamma$ of the non-trivial zeros of $\zeta(s)$.		
		\begin{enumerate}
			\item For $|t|\geq 4$ we have
			\begin{align}  \label{11_40am_04_05}
				\sum_{\gamma}\dfrac{1}{6+(t-\gamma)^2} \leq \dfrac{\log |t|}{2\sqrt{6}}.
			\end{align} 
			\item For $|t|\geq 100$ we have
			\begin{align} \label{11_41am_04_05}
				\sum_{\gamma}\dfrac{1}{\frac{1}{4}+(t-\gamma)^2} \leq 9\, {\log |t|}.
			\end{align} 
		\end{enumerate} 	
	\end{lemma}
	\begin{proof}
		Letting $s=\alpha+it$, and taking the real part of the fractional decomposition of $\zeta(s)$ (see \cite[Corollary 10.14]{MV}) one has
		\begin{align}  \label{10_35pm_28_04}
			\sum_{\rho}\dfrac{\alpha-\re{\rho}}{(\alpha-\re{\rho})^2+(t-\gamma)^2} & = \re\dfrac{\zeta'}{\zeta}(s)+\dfrac{1}{2}\,\re\dfrac{\Gamma'}{\Gamma}\bigg(\dfrac{s}{2}+1\bigg)-\dfrac{\log \pi}{2} + \dfrac{\alpha-1}{(\alpha-1)^2+t^2}.
		\end{align}
		Using the inequality $\re\frac{\Gamma'}{\Gamma}(z)\leq \log|z|$ for $\re{z}\geq \frac{1}{4}$ (see \cite[Lemma 2.3]{Chandee}), it follows that for $\alpha\geq 0$ and $t\neq 0$
		\begin{align*}
			\re\dfrac{\Gamma'}{\Gamma}\bigg(\dfrac{\alpha}{2}+1+\dfrac{it}{2}\bigg)\leq \log |t| - \log 2 + \frac{(\alpha+2)^2}{2t^2}.
		\end{align*}
		Since \replaced{RH}{Riemann hypothesis} holds, combining this and \eqref{10_35pm_28_04} we get
		\begin{align} \label{11_01pm_28_04}
			\sum_{\rho}\dfrac{\alpha-\hh}{(\alpha-\hh)^2+(t-\gamma)^2} & \leq \re\dfrac{\zeta'}{\zeta}(s)+\frac{\log |t|}{2} - \dfrac{\log 2\pi}{2} + \frac{(\alpha+2)^2}{4t^2} + \dfrac{\alpha-1}{(\alpha-1)^2+t^2}.
		\end{align}
		To prove \eqref{11_40am_04_05}, we let $\alpha = \sqrt{6}+\tfrac{1}{2}$ in 
		\eqref{11_01pm_28_04} and using the fact that $|\frac{\zeta'}{\zeta}(\sqrt{6}+\frac12+it)|\leq |\frac{\zeta'}{\zeta}(\sqrt{6}+\frac12)|= 0.1738\ldots$, and $|t|\geq4$ we conclude.	
		To prove \eqref{11_41am_04_05}, we let  $\alpha=1$ in 
		\eqref{11_01pm_28_04} and using that $|t|\geq 100$ we obtain
		\begin{align*} 
			\sum_{\gamma}\dfrac{1}{\frac{1}{4}+(t-\gamma)^2} \leq 2\,\re\dfrac{\zeta'}{\zeta}(1+it) + {\log |t|}\leq 2\left|\dfrac{\zeta'}{\zeta}(1+it)\right| + {\log |t|}
		\end{align*}
		Using Lemma \ref{7_23AM} with $\sigma=1$ we conclude.
		%	
		%	get \eqref{11_41am_04_05}.we 
		%$	$$
		%	\left|\dfrac{\zeta'}{\zeta}(\sigma+it)\right|\leq 4\log|t|.
		%	$$
	\end{proof}

	Throughout the paper, we will encounter situations where we aim to compute the integral of $|f+g|^2$, {where the $L^2$-norm of $f$ and $g$ by themselves are much easier to compute}. \replaced{To make this passage we shall use the following inequality:}{One could apply the Cauchy–Schwarz inequality to estimate this quantity; however, to streamline our work, a straightforward consequence of the AM–GM inequality leads to:} for any $\eta>0$
	\begin{align} \label{7_10_29_04}
		\int_{a}^{b}|f(x)+g(x)|^2\dx\leq (1+\eta)	\int_{a}^{b}|f(x)|^2\dx +  \left(1+\dfrac{1}{\eta}\right)\int_{a}^{b}|g(x)|^2\dx,
	\end{align} 
	which is an immediate consequence of the inequality $(x\sqrt{\eta}-\frac{y}{\sqrt{\eta}})^2\geq 0$.
	\smallskip
	
	\begin{lemma} \label{secondmoment} Assume \replaced{RH}{the Riemann hypothesis}. Then, for $T\geq 4\cdot 10^{13}$ we have
		$$
		\int_{10^4}^{T}\left|\dfrac{\zeta'}{\zeta}(1+it)\right|^2\!\!\dt\leq 0.8056\cdot T.
		$$
	\end{lemma}	
	\begin{proof}
		Given $x, y \geq 2$ and $s=1+it$ with $t\geq 10^4$, the unconditional formula \cite[Eq. (13.35)]{MV} states that 
		\begin{align} \label{identity}
			\dfrac{\zeta'}{\zeta}(s)=-\sum_{\rho}\dfrac{(xy)^{\rho-s}-x^{\rho-s}}{(\rho-s)^2\log y}-\sum_{k=1}^\infty\dfrac{(xy)^{-2k-s}-x^{-2k-s}}{(2k+s)^2\log y}+\dfrac{(xy)^{1-s}-x^{1-s}}{(1-s)^2\log y}-\sum_{n\leq xy}\dfrac{\Lambda(n)}{n^s}w(n),
		\end{align} 
		where $w(n)$ is a function that satisfies $0\leq w(n)\leq 1$.
		%	\[
		%	w(u) = w(x,y; u) =
		%	\begin{cases}
			%	1, & \text{if } 1 \leq u \leq x, \\[1pt]
			%	1 - \dfrac{\log(u/x)}{\log y}, & \text{if } x \leq u \leq xy, \\[1pt]
			%	0, & \text{if } u \geq xy.
			%\end{cases}
			%\]
			Let us bound each term on the right-hand side of \eqref{identity}. Since \replaced{RH}{the Riemann hypothesis} holds,  
			\begin{align*}
				%	\begin{split} \label{15_34pm}
					\left|\sum_{\rho}\dfrac{(xy)^{\rho-s}-x^{\rho-s}}{(\rho-s)^2\log y}\right| &  =\left|\sum_{\gamma}\dfrac{(xy)^{\frac{1}{2}+i\gamma-s}-x^{\frac{1}{2}+i\gamma-s}}{(\frac{1}{2}+i\gamma-s)^2\log y}  \right|  \leq \dfrac{x^{-\frac{1}{2}}(y^{-\frac{1}{2}}+1)}{\log y}\sum_{\gamma}\dfrac{1}{\frac{1}{4}+(t-\gamma)^2}.
				\end{align*}
				\replaced{Thus}{Then}, by Lemma \ref{12_30am} we arrive at
				$$	\left|\sum_{\rho}\dfrac{(xy)^{\rho-s}-x^{\rho-s}}{(\rho-s)^2\log y}\right| \leq \left(\dfrac{9x^{-\frac{1}{2}}(y^{-\frac{1}{2}}+1)}{\log y}\right) \log t:=c_{x,y}\,\log t.
				$$
				We estimate the next terms in \eqref{identity} trivially as follows \begin{align*}
					\Bigg|\sum_{k=1}^\infty\dfrac{(xy)^{-2k-s}-x^{-2k-s}}{(2k+s)^2\log y} \Bigg|\leq \dfrac{0.3}{t^2}, \,\,\,\,\,\,\,\,\,\,\,\mbox{and} \,\,\,\,\,\,\,\,\,\,\,\,
					\Bigg|\dfrac{(xy)^{1-s}-x^{1-s}}{(1-s)^2\log y}\Bigg| \leq \dfrac{2.9}{t^2}.
				\end{align*} 
				Inserting these bounds in \eqref{identity} we arrive at
				\begin{align*}
					\dfrac{\zeta'}{\zeta}(1+it) = & 
					-\sum_{n\leq xy}\dfrac{\Lambda(n)}{n^{1+it}}w(n)+	O^*\left(c_{x,y}\log t+\dfrac{3.2}{t^2}\right).
				\end{align*}
				for $t\geq 10^4$. Now, we integrate from $10^4$ to $T$ (with $10^4\leq T_0\leq T$), and by \eqref{7_10_29_04}, for any $\eta>0$:
				\begin{align}  \label{1_21am}
					\int_{10^4}^{T}\left|\dfrac{\zeta'}{\zeta}(1+it)\right|^2\!\!\dt & \leq  
					(1+\eta)\int_{10^4}^{T}\left|\sum_{n\leq xy}\dfrac{\Lambda(n)}{n^{1+it}}w(n)\right|^2\!\!\dt +	\left(1+\dfrac{1}{\eta}\right)O^*\left(\int_{10^4}^{T}\left|c_{x,y}\log t+ \frac{3.2}{t^2}\right|^2\!\!\dt\right).
				\end{align} 
				\replaced{Applying}{Now, from an application of} the explicit mean value \added{theorem} in \cite[Proposition 2.11]{Dona}\added{, we get} %\footnote{It is mentioned that if $\{a_n\}_{n\in\N}$ is a real sequence, the error term can be improved in \cite[Proposition 2.11]{Dona}. The improved result should be
					%$$
					%\int_{0}^T\left|\sum_{n\leq X}a_nn^{it}\right|^2\dt\leq \left(T+\dfrac{E}{4}\right)\sum_{n\leq X}|a_n|^2 + \dfrac{E}{2}\sum_{n\leq X}n|a_n|^2,
					%$$ where $E\leq 8.26495$.}
				$$
				\int_{10^4}^{T}\left|\sum_{n\leq xy}\dfrac{\Lambda(n)}{n^{1+it}}w(n)\right|^2\dt \leq \left(T-10^4 + 4.133\right) \sum_{n \leq xy}\left(\dfrac{\Lambda(n)}{n}w(n)\right)^2 +  8.265\sum_{n \leq xy} n \left(\dfrac{\Lambda(n)}{n}w(n)\right)^2.
				$$
				Since $|w(n)|\leq 1$, by (2) in Lemma \ref{Lemmaprime1}, the first sum is bounded by $0.8053$, and the second sum is bounded by $\sum_{n \leq xy}\frac{\Lambda^2(n)}{n}$. Thus, writing $xy=e^{\alpha}$, by Lemma \ref{Lemmaprime2} we conclude that
				\begin{align} \label{1_01am}
					\int_{10^4}^{T}\left|\sum_{n\leq xy}\dfrac{\Lambda(n)}{n^{1+it}}w(n)\right|^2\dt \leq 0.8053\,T + 4.1325\,\alpha^2 - 8\cdot 10^3.
				\end{align} 
				Moreover
				\begin{align}   \label{1_02am}
					\int_{10^4}^{T}\left|c_{x,y}\log t + \frac{3.2}{t^2}\right|^2\!\!\dt & \leq 2\int_{10^4}^{T}|c_{x,y}\log t|^2\dt+2\int_{10^4}^{T}\left|\frac{3.2}{t^2}\right|^2\!\!\dt <  2(c_{x,y})^2T\log ^2T +  7\cdot 10^{-12}.
				\end{align} 
				Letting $y=e^{2\lambda}$, with $\lambda\geq (\log 2)/2$, note that 
				\begin{align*}
					(1+\eta)4.1325\,\alpha^2 + \left(1+\dfrac{1}{\eta}\right)2(c_{x,y})^2T\log^2T = {4.1325\left(1+{\eta}\right)\alpha^2} + 40.5\left(1+\dfrac{1}{\eta}\right)\left(\dfrac{1+e^\lambda}{\lambda}\right)^2\dfrac{T\log^2T}{e^\alpha}.
				\end{align*} 
				In order to reduce the contribution from the above expression, we choose $\lambda=1.278$, and $\alpha=\log T$. Then, inserting \eqref{1_01am} and \eqref{1_02am} in \eqref{1_21am} we get
				\begin{align*}  
					\int_{10^4}^{T}\left|\dfrac{\zeta'}{\zeta}(1+it)\right|^2\!\!\dt & \leq  
					(1+\eta)0.8053T + \left({4.1325\left(1+{\eta}\right)} + 522.295\left(1+\dfrac{1}{\eta}\right)\right)\log^2T + \kappa_n,
				\end{align*} 
				where $\kappa_n=	-(1+\eta)8\cdot 10^3+(1+{\eta^{-1}})7\cdot 10^{-12}$.
				Finally, choosing $\eta=10^{-4}$, using $\kappa_n<0$, and $T\geq  4\cdot 10^{13}$, the proof is done.
			\end{proof}

			\begin{lemma} \label{11_18pm} 
				For $0< t\leq \frac{1}{2}$ we have the unconditional bound
				\begin{align*}
					\left|\frac{\zeta'}{\zeta}(1+it)+\frac{1}{it}\right|  \leq 2.635.
				\end{align*}
			\end{lemma}	
			\begin{proof}
				By the Laurent expansion
				\begin{align} \label{7pm}
					\zeta(s) = \frac{1}{s-1} + \sum_{n=0}^{\infty} \frac{(-1)^n \gamma_n}{n!} (s-1)^n,
				\end{align}
				together with the bound $
				|\gamma_n| \leq {4(n-1)!}/{\pi^n}$ for all $n\geq 2$ even, and $
				|\gamma_n| \leq {2(n-1)!}/{\pi^n}$ for all $n\geq 1$ odd (see \cite{Berndt}), we have for $0< t\leq \frac{1}{2}$ (letting $s=1+it$), 
				\begin{align*}
					\left|\zeta(1+it)-\dfrac{1}{it}\right|& \leq |\gamma_0|+ \sum_{\substack{n\,even \\ n\geq 2}}\dfrac{|\gamma_n|}{n!}t^n+\sum_{\substack{n\,odd \\ n\geq 1}}\dfrac{|\gamma_n|}{n!}t^n < 0.578+ 2\sum_{n=1}^\infty\dfrac{1}{n(2\pi)^n}+2\sum_{\substack{n\,even \\ n\geq 2}}\dfrac{1}{n(2\pi)^n} \\
					& = 0.578-2\log\left(1-\dfrac{1}{2\pi}\right)-\log\left(1-\dfrac{1}{4\pi^2}\right)<0.951.
				\end{align*}
				Moreover, differentiating \eqref{7pm}, we bound similarly as before to get
				$$
				\left|\zeta'(1+it)-\dfrac{1}{t^2}\right|\leq 4\sum_{n=1}^\infty\dfrac{1}{(2\pi)^n} + 4\sum_{\substack{n\,even \\ n\geq 2}}\dfrac{1}{(2\pi)^n} = \dfrac{4}{2\pi-1}+\dfrac{4}{4\pi^2-1}<0.862.
				$$
				Thus
				\begin{align*}
					\frac{\zeta'}{\zeta}(1+it)+\frac{1}{it}  = \dfrac{1/t^2 +O^*(0.862)}{1/it+O^*(0.951)} + \dfrac{1}{it} = \dfrac{O^*(0.862t)+O^*(0.951)}{1+O^*(0.951t)}.
				\end{align*}
				Therefore, for $0< t\leq\frac{1}{2}$:
				\begin{align*}
					\left|\frac{\zeta'}{\zeta}(1+it)+\frac{1}{it}\right|  \leq \dfrac{0.862/2 + 0.951}{1-0.951/2}\leq 2.635.
				\end{align*}
			\end{proof}

			\begin{lemma} \label{4_24pm} For any $0<\delta\leq 1$ and any $t\neq 0$ we have the bound
				$$
				\left|\dfrac{(1+\delta)^{\frac{1}{2}+it}-1}{\frac{1}{2}+it}\right|\leq \min\bigg\{\delta,\dfrac{\ell}{|t|}\bigg\},
				$$
				where $\ell=\sqrt{1+\delta}+1$.
				In particular, \replaced{assuming RH}{under the Riemann hypothesis}, $$
				\left|\dfrac{(1+\delta)^\rho-1}{\rho}\right|\leq \min\bigg\{\delta,\dfrac{\ell}{|\gamma|}\bigg\}\replaced{,}{.}
				$$ \added{for any non-trivial zero $\rho$.}
			\end{lemma}
			\begin{proof}
				Clearly
				$$
				\left|\dfrac{(1+\delta)^{\frac{1}{2}+it}-1}{\frac{1}{2}+it}\right|=\left|\int_{1}^{1+\delta}x^{-\frac{1}{2}+it}\dx\right|\leq \int_{1}^{1+\delta}x^{-\frac{1}{2}}\dx \leq \delta.
				$$
				On the other hand \added{we also have},
				$$	\left|\dfrac{(1+\delta)^{\frac{1}{2}+it}-1}{\frac{1}{2}+it}\right|= \left|\dfrac{e^{(\frac{1}{2}+it)\log(1+\delta)}-1}{\frac{1}{2}+it}\right| \leq \dfrac{e^{\frac{1}{2}{\log(1+\delta)}}+1}{|t|} = \dfrac{\sqrt{1+\delta}+1}{|t|},
				$$
				\added{which gives the desired conclusion.}
			\end{proof}

			\begin{lemma} \label{9_52pm} We have the following estimates for $T\geq 10^{13}$:
				$$\sum_{0<\gamma\leq T}\log \gamma\leq \dfrac{1}{2\pi}\cdot T\log^2T, \,\,\,\,\,\,\,\, \mbox{and} \,\,\,\,\,\,\,\,\, \sum_{\gamma> T}\dfrac{\log \gamma}{\gamma^2}\leq \dfrac{1.028}{2\pi}\cdot\dfrac{\log^2T}{T},$$
				where the sums run over the imaginary parts $\gamma$ of the non-trivial zeros of $\zeta(s)$.		
			\end{lemma}
			\begin{proof}
				To prove the first estimate, we apply \cite[Lemma 3]{Trudgian1} with $\phi(t)=\log t$, $T_1=2\pi e$, $T_2=T$, $A=0.28$, to get
				$$
				\sum_{2\pi e<\gamma \leq T}\log \gamma\leq  \dfrac{1}{2\pi}\int_{2\pi e}^{T}\log t\log\left(\dfrac{t}{2\pi}\right)\dt + 0.56\log^2T + 0.28\int_{2\pi e}^{T}\dfrac{\log t}{t}\dt.
				$$
				Making the computations, using the facts that $T\geq 10^{13}$, $\gamma_1=14.1347\ldots$ and $\gamma_2=21.0220\ldots$ we conclude. To prove the second estimate, we apply \cite[Lemma 5 and Lemma 6]{Trudgian1} to get
				$$\sum_{\gamma> T}\dfrac{\log \gamma}{\gamma^2}\leq \dfrac{1}{2\pi}\left(\dfrac{\log^2T}{T} + \dfrac{\ln(2\pi/e)\log T}{T}\right).$$
				This implies the desired result.
			\end{proof}
			
			%	The following lemma is somewhat technical, but it will help simplify our computations.
			%	\begin{lemma} 
				%	Let $f,g:[a,b]\to\C$ two integrable functions. Then, for any $\eta>0$ we have
				%	$$
				%		\int_{a}^{b}|f(x)+g(x)|^2\dx\leq (1+\eta)	\int_{a}^{b}|f(x)|^2\dx +  \left(1+\dfrac{1}{\eta}\right)\int_{a}^{b}|g(x)|^2\dx.
				%	$$
				%\end{lemma}
				%\begin{proof}
				%	This result is a straightforward consequence of the MA–MG inequality.
				%\end{proof}
				
				\vspace{0.5cm}

				\section{An explicit bound for $J_\psi(x,\delta)$} \label{10_39pm}
				
				To derive an explicit bound for $J_\theta(x,\delta)$, we begin by estimating the integral defined in~\eqref{10_48pm}, employing the averaging technique introduced by Saffari and Vaughan in \cite{SV}. 
				
				\begin{theorem}\label{psiMeanValue}
					Assume \replaced{RH}{the Riemann hypothesis}. Then, for $x\geq 10^{13}$ and $\delta\in (0,10^{-13}]$ we have that
					\begin{align*}  
						J_{\psi}(x,\delta)
						\leq  2.2258\cdot\delta\,\log^2\left(\dfrac{2.0001}{\delta}\right)x^2.
					\end{align*}
				\end{theorem} 
				\begin{proof}
					Let $\lambda >1$ and $\kappa>1$ be two parameters to be chosen later. For any $x>0$, note that $[x, \kappa x]\subset [x\nu/\lambda , \kappa x\nu]$
					for $1\leq \nu\leq \lambda$. This implies that
					\begin{align}  \label{11_24pm}
						\int_{x}^{\kappa x}\big(\psi((1+\delta)y)-\psi(y)-\delta y\big)^2\dy\leq \dfrac{1}{(\lambda-1)}\int_{1}^{\lambda}\left(\int_{x\nu/\lambda }^{\kappa x\nu}\big(\psi((1+\delta)y)-\psi(y)-\delta y\big)^2\dy\right)\d\nu.
					\end{align}
					Let us concentrate on bounding the double integral on the right-hand side of \eqref{11_24pm}. For all $y>0$, $y\notin\Z$ we have the explicit formula \cite[Eq. (12.1)]{MV} given by
					\begin{align} \label{explicitformula}
						\psi(y)=y-\lim_{T\to \infty}\sum_{|\gamma|\leq T}\dfrac{y^{\rho}}{\rho}-\dfrac{\zeta'(0)}{\zeta(0)}-\dfrac{1}{2}\log(1-y^{-2}).
					\end{align} 
					We write $\psi((1+\delta)y)-\psi(y)-\delta y =A_\delta(y) + B_\delta(y)$, where
					$$
					\,\,\,\,\,\,\, \,\,\,\,\, A_\delta(y)=-\lim_{T\to\infty}\displaystyle\sum_{|\gamma|\leq T}\dfrac{(1+\delta)^\rho-1}{\rho}y^\rho, \,\,\,\,\,\,\,\,\, \mbox{and} \,\,\,\,\,\,\,\,\,
					B_\delta(y)=-\dfrac{1}{2}\left(\log(1-((1+\delta)y)^{-2})-\log(1-y^{-2})\right).
					$$
					By \ref{7_10_29_04} we see that 
					\begin{align}  \label{10_15pm}
						\begin{split}
							\int_{1}^{\lambda} &\left(\int_{x\nu/\lambda}^{\kappa x\nu}\big(\psi((1+\delta)y)- \psi(y)-\delta y\big)^2\dy\right)\d\nu \\
							& 
							\,\,\,\,\,\,\,\,\,\,\,\,\,\,\,\,\,\,\,\,\,\,\,\,\,\,\,\,\,\,\,\,\,\,\,\,\,\,\,\,\,\,\leq  (1+\eta) \int_{1}^{\lambda}\left(\int_{x\nu/\lambda}^{\kappa x\nu}\left|A_\delta(y)\right|^2\dy\right)\dnu + \left(1+\dfrac{1}{\eta}\right)\int_{1}^{\lambda}\left(\int_{x\nu/\lambda}^{\kappa x\nu}\left|B_\delta(y)\right|^2\dy\right)\dnu. 
						\end{split}
					\end{align} 
					Let us analyze the double integral of $A_\delta(y)$ in the above expression. Since RH holds, we write $\rho_1=\hh+i\gamma_1$ and $\rho_2=\hh+i\gamma_2$. Clearly
					\begin{align*}
						\int_{x\nu/\lambda}^{\kappa x\nu}\left|A_\delta(y)\right|^2\dy = \sum_{\rho_1}\sum_{\rho_2} \bigg(\dfrac{(1+\delta)^{\rho_1}-1}{\rho_1}\bigg)\bigg(\dfrac{(1+\delta)^{\overline{\rho_2}}-1}{{\overline{\rho_2}}}\bigg)\bigg(\dfrac{\kappa^{1+\rho_1+\overline{\rho_2}}-(1/\lambda)^{1+\rho_1+\overline{\rho_2}}}{1+\rho_1+\overline{\rho_2}}\bigg)(x\nu)^{1+\rho_1+\overline{\rho_2}},
					\end{align*}
					by dominated convergence theorem, since the double sum 
					$$
					\sum_{\gamma_1}\sum_{\gamma_2}\dfrac{1}{|\gamma_1|}\dfrac{1}{|\gamma_2|}\dfrac{1}{2+|\gamma_1-\gamma_2|}
					$$
					is bounded (Lemma \ref{4_24pm} and \cite[Eq. (9)]{Trudgian1}). 
					% Note that  on the right-hand side converges absolutely and uniformly in $\nu$ on $[1,\lambda]$. In fact, letting $\rho_1=\hh+i\gamma_1$ and $\rho_2=\hh+i\gamma_2$,  one can show that
					%		$$
					%		\sum_{\gamma_1}\sum_{\gamma_2}\dfrac{1}{|\gamma_1|}\dfrac{1}{|\gamma_2|}\dfrac{1}{2+|\gamma_1-\gamma_2|}<\infty.
					%		$$
					%See for instance Lemma \ref{4_24pm} or \cite[Eq. (9)]{Trudgian1}. 
					Now, integrating over $\nu$ we have that
					\begin{align*}
						&\int_{1}^{\lambda}\left(\int_{x\nu/\lambda}^{\kappa x\nu}\left|A_\delta(y)\right|^2\dy\right)\dnu\\
						& = \sum_{\rho_1}\sum_{\rho_2} \bigg(\dfrac{(1+\delta)^{\rho_1}-1}{\rho_1}\bigg)\bigg(\dfrac{(1+\delta)^{\overline{\rho_2}}-1}{{\overline{\rho_2}}}\bigg)\bigg(\dfrac{\kappa^{1+\rho_1+\overline{\rho_2}}-(1/\lambda)^{1+\rho_1+\overline{\rho_2}}}{1+\rho_1+\overline{\rho_2}}\bigg)\bigg(\dfrac{\lambda^{2+\rho_1+\overline{\rho_2}}-1}{2+\rho_1+\overline{\rho_2}}\bigg)x^{1+\rho_1+\overline{\rho_2}}.
					\end{align*}
					Note that
					$|1+\rho_1+\overline{\rho_2}||2+\rho_1+\overline{\rho_2}|=\sqrt{6^2+13(\gamma_1-\gamma_2)^2+(\gamma_1-\gamma_2)^4}\geq 6+(\gamma_1-\gamma_2)^2$. \replaced{We then}{Then,} \replaced{use}{using} Lemma \ref{4_24pm}, \replaced{RH}{the Riemann hypothesis} and the estimate $2|ab|\leq |a|^2+|b|^2$ \replaced{to conclude}{it follows} that 
					\begin{align*} 
						\int_{1}^{\lambda}\left(\int_{x\nu/\lambda}^{\kappa x\nu}\left|A_\delta(y)\right|^2\dy\right)\dnu  & \leq \left(\kappa^2+\dfrac{1}{\lambda^2}\right)(\lambda^3+1) \displaystyle\sum_{\gamma_1}\displaystyle\sum_{\gamma_2} 
						\min\bigg\{\delta,\dfrac{\ell}{|\gamma_1|}\bigg\}\min\bigg\{\delta,\dfrac{\ell}{|\gamma_2|}\bigg\}\dfrac{1}{6+(\gamma_1-\gamma_2)^2}\,x^2 \\
						&	\leq \left(\kappa^2+\dfrac{1}{\lambda^2}\right)(\lambda^3+1) \displaystyle\sum_{\gamma_1} 
						\min\bigg\{\delta^2,\dfrac{\ell^2}{|\gamma_1|^2}\bigg\}\displaystyle\left(\sum_{\gamma_2}\dfrac{1}{6+(\gamma_1-\gamma_2)^2}\right)x^2.
					\end{align*}
					%		Here we have used that $|1+\rho_1+\overline{\rho_2}||2+\rho_1+\overline{\rho_2}|=\sqrt{6^2+13(\gamma_1-\gamma_2)^2+(\gamma_1-\gamma_2)^4}\geq 6+(\gamma_1-\gamma_2)^2$. 
					Using that $|\gamma|>14$, Lemma \ref{12_30am} and the symmetry of the zeros we have that
					\begin{align*} 
						\int_{1}^{\lambda}\left(\int_{x\nu/\lambda}^{\kappa x\nu}\left|A_\delta(y)\right|^2\dy\right)\dnu  & \leq  \frac{1}{2\sqrt{6}}\left(\kappa^2+\dfrac{1}{\lambda^2}\right)(\lambda^3+1) \displaystyle\sum_{\gamma} 
						\min\bigg\{\delta^2,\dfrac{\ell^2}{|\gamma|^2}\bigg\}\log|\gamma|x^2 \\
						& =\frac{1}{\sqrt{6}}\left(\kappa^2+\dfrac{1}{\lambda^2}\right)(\lambda^3+1) \left(\displaystyle\sum_{\gamma>0} 
						\min\bigg\{\delta^2,\dfrac{\ell^2}{\gamma^2}\bigg\}\log\gamma\right)x^2 \\
						& = \frac{1}{\sqrt{6}}\left(\kappa^2+\dfrac{1}{\lambda^2}\right)(\lambda^3+1) \left(\delta^2\displaystyle\sum_{0<\gamma\leq \ell/\delta} 
						\log\gamma+\ell^2\displaystyle\sum_{\gamma>\ell/\delta}
						\dfrac{\log\gamma}{{\gamma^2}}\right)x^2.
					\end{align*}
					Since $0<\delta\leq 10^{-13}$ we see that $\ell/\delta \geq 2\cdot 10^{13}$. Applying Lemma \ref{9_52pm} we arrive at
					\begin{align}  \label{11_19pm}
						\begin{split} 
							\int_{1}^{\lambda}\left(\int_{x\nu/\lambda}^{\kappa x\nu}\left|A_\delta(y)\right|^2\dy\right)\dnu  & \leq \dfrac{2.028}{2\sqrt{6}\pi}\left(\kappa^2+\dfrac{1}{\lambda^2}\right)(\lambda^3+1) \,\delta\,\ell\,\log^2\left(\dfrac{\ell}{\delta}\right)x^2\\
							& <\dfrac{2.0282}{\sqrt{6}\pi}\left(\kappa^2+\dfrac{1}{\lambda^2}\right)(\lambda^3+1) \,\delta\,\log^2\left(\dfrac{\ell}{\delta}\right)x^2,
						\end{split} 
					\end{align}
					where we used that $\ell<2.0001$.
					
					Now, let us analyze the double integral of $B_\delta(y)$. By the mean value theorem we get
					% Define the function $g(t)=\log(1-t^{-2})$. Using the mean value theorem we have for $y>1$ that 
					%$
					%g((1+\delta)y)-g(y)= \delta yg'(\xi)
					%$, for some $\xi\in (y,(1+\delta)y)$. Since $g'(t)=2/t(t^2-1)$ we get
					$$
					B_\delta(y)=\dfrac{1}{2}\left|\log(1-((1+\delta)y)^{-2})-\log(1-y^{-2})\right|\leq \dfrac{\delta y}{y(y^2-1)}\leq  \dfrac{2\delta}{y^2}
					$$
					where we have assumed that $y\geq \sqrt{2}$. Therefore, for $x\geq \sqrt{2}\lambda$ we have that
					\begin{align}  \label{11_19pm2}
						\begin{split} 
							\int_{1}^{\lambda}\left(\int_{x\nu/\lambda}^{\kappa x\nu}\left|B_\delta(y)\right|^2\dy\right)\dnu& \leq  \int_{1}^{\lambda}\left(\int_{x\nu/\lambda}^{\kappa x\nu}\dfrac{4\delta^2}{y^4}\dy\right)\dnu  = \dfrac{2}{3}\left(\lambda^3-\dfrac{1}{\kappa^3}\right)\left(1-\dfrac{1}{\lambda^2}\right)\delta^2x^{-3} \\
							& 	= \dfrac{2}{3}\left(\lambda^3-\dfrac{1}{\kappa^3}\right)\left(1-\dfrac{1}{\lambda^2}\right)\left(\dfrac{\delta x^{-5}}{\log^2(\ell/\delta)}\right)\delta\,\log^2\left(\dfrac{\ell}{\delta}\right)x^2 \\
							& < 1.3\cdot 10^{-17}\left(\lambda^3-\dfrac{1}{\kappa^3}\right)\left(1-\dfrac{1}{\lambda^2}\right)\delta\,\log^2\left(\dfrac{\ell}{\delta}\right)x^2,
							%\dfrac{(8\lambda^3-1)(\lambda^2-1)}{12\lambda^2}\delta^2x^{-3}.
						\end{split}
					\end{align} 
					where we used that 	$\delta\mapsto \ell/\delta$ is a decreasing function for $\delta>0$ and $x>\sqrt{2}$. Combining \eqref{11_19pm} and \eqref{11_19pm2} in \eqref{10_15pm} and then in \eqref{11_24pm} we get that for $\lambda>1$, $\kappa>1$, $\eta>0$ and $x\geq \sqrt{2}\lambda$:
					\begin{align}  \label{11_41pm}
						\int_{x}^{\kappa x}&\big(\psi((1+\delta)y)-\psi(y)-\delta y\big)^2\dy \leq \mathcal{C}(\kappa, \lambda, \eta)\cdot \delta\,\log^2\left(\dfrac{\ell}{\delta}\right)x^2,
					\end{align}
					where
					$$
					\mathcal{C}( \kappa, \lambda, \eta)=(1+\eta)\dfrac{2.0282}{\sqrt{6}\pi}\left(\kappa^2+\dfrac{1}{\lambda^2}\right)\left(\dfrac{\lambda^3+1}{\lambda-1}\right) +  1.3\cdot 10^{-17}\left(1+\frac{1}{\eta}\right)\left(\lambda^3-\dfrac{1}{\kappa^3}\right)\left(\dfrac{\lambda +1}{\lambda^2}\right).
					$$
					To reduce the notation, let us write $\mathcal{V}(\delta,y)=\psi((1+\delta)y)-\psi(y)-\delta y$. Consider $x\geq \max(10^{13},\sqrt{2}\lambda)$, and let $N\geq 0$ be the integer such that ${x}/{\kappa^{N+1}}<\sqrt{2}\lambda\leq {x}/{\kappa^{N}}$. Then,
					\begin{align} \label{3_12pm}
						\begin{split} 
							\int_{1}^{x}(\mathcal{V}(\delta,y))^2\dy & \leq \sum_{n=0}^{N}\int_{x/\kappa^{n+1}}^{x/\kappa^n}(\mathcal{V}(\delta,y))^2\dy  + \int_{1}^{\sqrt{2}\lambda}(\mathcal{V}(\delta,y))^2\dy.
						\end{split}
					\end{align} 
					Since $\mathcal{V}(\delta,y)=0$ for $y<\frac{2}{1+\delta}$, we bound the last term in the above sum as follows,
					\begin{align} \label{3_03pm}
						\begin{split}
							\int_{x/\kappa^{N+1}}^{x/\kappa^N}(\mathcal{V}(\delta,y))^2\dy & \leq 	\int_{\frac{2}{1+\delta}}^{\sqrt{2}\kappa\lambda}(\mathcal{V}(\delta,y))^2\dy   \leq \sum_{j=2}^{[\sqrt{2}\kappa\lambda]+1}
							\int_{\frac{j}{1+\delta}}^{j}(\mathcal{V}(\delta,y))^2\dy  + \sum_{j=2}^{[\sqrt{2}\kappa\lambda]}\int_{j}^{\frac{j+1}{1+\delta}}(\mathcal{V}(\delta,y))^2\dy.
						\end{split}
					\end{align}
					Let us bound each sum in the above expression. By \cite[Theorem 10]{Sch} we know the bound 
					$|\psi(y)-y|\leq\frac{1}{8\pi}{\sqrt{y}\log^2y}$ for $y\ge 73.2$. A numerical computation for the cases less than $73.2$ shows the estimate $|\psi(y)-y|\leq{2\sqrt{y}\log^2y}$ for $y\ge\frac{2}{1+\delta}$. Therefore, $(\mathcal{V}(\delta,y))^2\leq 2(\psi((1+\delta)y)-(1+\delta)y)^2+ 2(\psi(y)-y)^2\leq 16.01y\log^4y$ for $y\geq \frac{2}{1+\delta}$.
					%\leq 8y\log^4y\Big((1+\delta)\left(1+\frac{\log(1+\delta)}{\log y}\right)^4+1\Big)\leq 16.01y\log^4y$
					This implies that
					\begin{align*}
						\sum_{j=2}^{[\sqrt{2}\kappa\lambda]+1}
						\int_{\frac{j}{1+\delta}}^{j}(\mathcal{V}(\delta,y))^2\dy  & \leq\sum_{j=2}^{[\sqrt{2}\kappa\lambda]+1}\left(j-\dfrac{j}{1+\delta}\right)\sup_{y\in[2/(1+\delta),[\sqrt{2}\kappa\lambda]+1]}(\mathcal{V}(\delta,y))^2 \\
						& \leq 16.01([\sqrt{2}\kappa\lambda]+1)\log^4([\sqrt{2}\kappa\lambda]+1)\sum_{j=2}^{[\sqrt{2}\kappa\lambda]+1}\left(j-\dfrac{j}{1+\delta}\right) \\
						& =16.01([\sqrt{2}\kappa\lambda]+1)\log^4([\sqrt{2}\kappa\lambda]+1) \dfrac{([\sqrt{2}\kappa\lambda]+3)[\sqrt{2}\kappa\lambda]}{2}\dfrac{\delta}{1+\delta}: = \dfrac{\alpha(\kappa,\lambda)}{1+\delta}\cdot\delta.
					\end{align*} 
					On the other hand, note that for $j\leq y\leq \frac{j+1}{1+\delta}$, we have that $\psi((1+\delta)y)=\psi(y)$. Then
					$$
					\sum_{j=2}^{[\sqrt{2}\kappa\lambda]}\int_{j}^{\frac{j+1}{1+\delta}}(\mathcal{V}(\delta,y))^2\dy = \sum_{j=2}^{[\sqrt{2}\kappa\lambda]}\int_{j}^{\frac{j+1}{1+\delta}}(\delta y)^2\dy\leq \dfrac{([\sqrt{2}\kappa\lambda]+1)^3}{3}\cdot \delta^2:= \beta(\kappa,\lambda)\cdot \delta^2.
					$$
					Inserting these bounds in \eqref{3_03pm}, 
					\begin{align*}
						\int_{x/\kappa^{N+1}}^{x/\kappa^N}(\mathcal{V}(\delta,y))^2\dy\leq \dfrac{\alpha(\kappa,\lambda)}{1+\delta}\cdot\delta + \beta(\kappa,\lambda)\cdot \delta^2. 
					\end{align*}
					Now, assume that $\lambda<2$. Applying \eqref{11_41pm} for $x\geq 10^{13}$, we see that\footnote{If $N=0$ this sum is empty.}
					\begin{align*}
						\sum_{n=0}^{N-1}\int_{x/\kappa^{n+1}}^{x/\kappa^n}(\mathcal{V}(\delta,y))^2\dy &  \leq \mathcal{C}(\kappa,\lambda,\eta) \cdot\delta\,\log^2\left(\dfrac{\ell}{\delta}\right)\sum_{n=0}^{N-1}\left(\dfrac{x}{\kappa^{n+1}}\right)^2   \leq  \dfrac{\mathcal{C}(\kappa,\lambda,\eta)}{\kappa^2-1} \cdot\delta\,\log^2\left(\dfrac{\ell}{\delta}\right)x^2.
					\end{align*}
					Finally,
					\begin{align*}
						\int_{1}^{\sqrt{2}\lambda}(\mathcal{V}(\delta,y))^2\dy & \leq 	\int_{1}^{2.9}(\mathcal{V}(\delta,y))^2\dy  \leq \int_{\frac{2}{1+\delta}}^{2}(\mathcal{V}(\delta,y))^2\dy+\int_{2}^{\frac{3}{1+\delta}}(\mathcal{V}(\delta,y))^2\dy\\
						& \leq \dfrac{2(\log 2)^2\delta}{1+\delta} +\int_{2}^{\frac{3}{1+\delta}}\delta^2 y^2\dy < 0.961\delta.
					\end{align*}
					Combining the previous bounds in \eqref{3_12pm} we conclude for $1<\lambda<2$, $\kappa>1$, $\eta>0$ and $x\geq 10^{13}$,
					\begin{align*} 
						\int_{1}^{x}\big(\psi((1+\delta)y)-\psi(y)-\delta y\big)^2\dy & \leq \dfrac{\mathcal{C}(\kappa,\lambda,\eta)}{\kappa^2-1} \cdot\delta\,\log^2\left(\dfrac{\ell}{\delta}\right)x^2 + \left(\dfrac{\alpha(\kappa,\lambda)}{1+\delta} + 0.961+ \beta(\kappa,\lambda)\delta \right) \delta. 
					\end{align*} 
					Minimizing the expression ${\mathcal{C}(\kappa,\lambda,\eta)}/{(\kappa^2-1)}$ for $1<\lambda<2$, $\kappa>1$, $\eta>0$, we choose $\kappa=100$, $\lambda=1.677$ y $\eta=5\cdot 10^{-11}$. Then 
					$$
					\dfrac{\mathcal{C}(\kappa,\lambda,\eta)}{\kappa^2-1}= 2.22571\ldots, \,\,\,\, \dfrac{\alpha(\kappa,\lambda)}{1+\delta}\leq 9.8\cdot 10^{10}, \,\,\,\, \mbox{and} \,\,\,\, \beta(\kappa,\lambda)\leq 4.5\cdot 10^6.
					$$
					Since $0<\delta\leq 10^{-13}$, $\delta\mapsto \ell/\delta$ is a decreasing function for $\delta>0$, and $x\geq 10^{13}$ we conclude.
				\end{proof}
				
				\vspace{0.5cm}
				
				\section{An explicit bound for $J_\theta(x,\delta)$} \label{10_39pm22}
				
				As we mentioned in the introduction, we want an explicit bound for $J_\theta(x,\delta)$. Our goal in this section is to establish the following result.
				
				\begin{theorem}\label{thetaMeanValue}
					Assume \replaced{RH}{the Riemann hypothesis}. Then, for $x\geq 10^{13}$ and $\delta\in (0,10^{-13}]$ we have that\footnote{We highlight that one can prove that $J_{\theta}(x,\delta)
						\leq  2.2259\cdot\delta\,\log^2\left(\frac{2.0001}{\delta}\right)x^2$, for $x$ sufficiently large and $\delta$ sufficiently small.}
					\begin{align*}  
						J_{\theta}(x,\delta)
						\leq  2.5571\cdot\delta\,\log^2\left(\dfrac{2.0001}{\delta}\right)x^2.
					\end{align*}
				\end{theorem} 
				The proof relies on comparing the integrals $J_{\theta}(x,\delta)$ and $J_{\psi}(x,\delta)$, and noting that their difference is a negligible error term. \added{At first such a passage sounds trivial: the difference is a sum supported on prime powers $p^{\ell}$ with $\ell\geq 2$, and should thus be negligable straight away. However, since we are working in the short interval $[y,(1+\delta)y]$, this naive approach seems to fall short. To bound the difference we} \deleted{We} follow \added{the method} Saffari–Vaughan\deleted{’s} \added{ presented in \cite[p. 22]{SV}} \deleted{method} only partially\deleted{, as presented in}. In their approach, the mentioned error term is bounded by  $O(\delta\,\log^2\left(\tfrac{1}{\delta}\right)x^2)$, i.e. the same as the main term. In our case, however, a more refined estimation of the error term is required\added{, to get the sharpest constant possible.} 
				
				%	\markus{I think we should add a comment here about what we instead bound the error term by.}

				\subsection{Proof of Theorem \ref{thetaMeanValue}: first step}
				\replaced{We start by applying \eqref{7_10_29_04} two times.}{The proof of Theorem \ref{thetaMeanValue} starts \added{by} applying Lemma \ref{7_10_29_04} two times.} Thus, for any $\eta>0$
				\begin{align*}
					J_\theta(x,\delta)
					& \leq \left(1+{\eta}\right)J_{\psi}(x,\delta) + \left(1+\frac{1}{\eta}\right)\int_1^x \big(\psi((1+\delta)y) - \psi(y)-\theta((1+\delta)y)+\theta(y)\big)^2\dy  \\
					&\leq \left(1+{\eta}\right)J_{\psi}(x,\delta)  +\left(1+\frac{1}{\eta}\right)^2\int_1^x \left((1+\delta)^{\frac{1}{2}}y^{\frac{1}{2}}-y^{\frac{1}{2}}\right)^2 \dy  \\
					& \,\,\,\,\,\,\,\,\,\,\,\,\,+\dfrac{(1+\eta)^2}{\eta}\int_1^x \left(\psi((1+\delta)y)-\psi(y) - \theta((1+\delta)y)+\theta(y) - (1+\delta)^{\frac{1}{2}}y^{\frac{1}{2}}+y^{\frac{1}{2}}\right)^2 \dy\\
					&\leq \left(1+{\eta}\right)J_{\psi}(x,\delta)  +\left(1+\dfrac{1}{\eta}\right)^2\dfrac{\delta^2x^2}{2} \\
					& \,\,\,\,\,\,\,\,\,\,\,\,\,+\dfrac{(1+\eta)^2}{\eta}\int_1^x \left(\psi((1+\delta)y)-\psi(y) - \theta((1+\delta)y)+\theta(y) - (1+\delta)^{\frac{1}{2}}y^{\frac{1}{2}}+y^{\frac{1}{2}}\right)^2 \dy.
				\end{align*}
				Thus, we need to bound the last integral on the right-hand side of the above expression. By a change of variable $y=e^\nu$, we have that this integral is exactly
				$$
				\int_0^{\log x} |\Delta_\delta(\nu)e^{\nu}|^2\dnu,
				$$
				where $\Delta_\delta(\nu)\coloneq \left(\psi((1+\delta)e^{\nu})-\psi(e^{\nu})-\theta((1+\delta)e^{\nu})+\theta(e^{\nu})-(1+\delta)^{\frac{1}{2}}e^{\frac{\nu}{2}}+e^{\frac{\nu}{2}}\right)e^{-\frac{\nu}{2}}$.
				Therefore, 
				\begin{align} \label{1_49am_30_04}
					\begin{split} 
						J_\theta(x,\delta)
						& \leq \left(1+{\eta}\right)J_{\psi}(x,\delta) + \left(1+\dfrac{1}{\eta}\right)^2\dfrac{\delta^2x^2}{2}+\dfrac{(1+\eta)^2}{\eta}\int_0^{\log x} |\Delta_\delta(\nu)e^{\nu}|^2\dnu \\
						&\leq  \left(1+{\eta}\right)J_{\psi}(x,\delta) + \left(1+\dfrac{1}{\eta}\right)^2\dfrac{\delta^2x^2}{2}+\dfrac{(1+\eta)^2}{\eta}\,x^2\int_0^{\infty} |\Delta_\delta(\nu)|^2\dnu. 
					\end{split} 
				\end{align}
				To bound the last integral in \eqref{1_49am_30_04}, we \added{shall} use Plancherel's theorem. \replaced{By Perron's formula we have}{In fact}, for $y\geq 1$ and $y \not \in \mathbb{Z}$, \replaced{that}{we have} (see \cite[p. 22]{SV})
				\begin{align} \label{10_30pm}
					\psi(y)-\theta(y)-y^{\frac{1}{2}}+1 = \frac{1}{2\pi i}\lim_{T\to \infty} \int_{\frac{1}{2}-iT}^{\frac{1}{2}+iT} \left(-\frac{\zeta'}{\zeta}(2s)-\frac{1}{2s-1}+\sum_p \frac{\log p}{p^s(p^{2s}-1)}\right)\frac{y^s}{s}\ds.
				\end{align}
				Letting 
				\begin{align} \label{11_37pm}
					F(t)\coloneq -\frac{\zeta'}{\zeta}(1+2it)-\frac{1}{2it}+\sum_p \frac{\log p}{p^{\frac{1}{2}+it}(p^{1+2it}-1)},
				\end{align} 
				by \eqref{10_30pm}, we have the following equality for almost every $\nu\geq 0$,
				$$
				\Delta_\delta(\nu) = \frac{1}{2\pi} \lim_{T\to \infty}\int_{-T}^T F(t)\left(\dfrac{(1+\delta)^{\frac{1}{2}+it}-1}{\frac{1}{2} +it}\right)e^{\nu it}\dt.
				$$
				By Lemma \ref{7_23AM} and \eqref{7pm}, $F$ is continuous on $\R$ and $F(t)=O(\log t)$. Then, we conclude that the integrand belongs to $L^2(\R)$. By Fourier inversion formula and Plancherel's theorem we obtain that
				$$\int_0^\infty |\Delta_\delta(\nu)|^2\dnu \leq \frac{1}{2\pi} \int_{-\infty}^\infty \left|F(t)\left(\dfrac{(1+\delta)^{\frac{1}{2}+it}-1}{\frac{1}{2} +it}\right)\right|^2\dt.$$
				\replaced{We now split up the integral on the right hand side.}{Then, by} \added{By} Lemma \ref{4_24pm} and the fact that $|F(t)|=|F(-t)|$ we get
				\begin{align*}
					\int_0^\infty |\Delta_\delta(\nu)|^2\dnu \leq  \frac{\delta^2}{\pi} \int_{0}^{\frac{\ell}{\delta}} \left|F(t)\right|^2\dt +  \frac{\ell^2}{\pi} \int_{\frac{\ell}{\delta}}^\infty\dfrac{|F(t)|^2}{t^2}\dt.  
				\end{align*}
				Therefore, by Theorem \ref{psiMeanValue} in \eqref{1_49am_30_04} we arrive at \begin{align} \label{1_57am_30_04}
					\begin{split}
						\!\!J_\theta(x,\delta)
						\leq  & \left(1+{\eta}\right)\cdot2.2258\cdot\delta\,\log^2\left(\dfrac{2.0001}{\delta}\right)x^2 + \left(1+\dfrac{1}{\eta}\right)^2\dfrac{\delta^2x^2}{2}\\
						& \,\,\,\,\,\, +\dfrac{(1+\eta)^2\delta^2x^2}{\eta\pi}\int_{0}^{\frac{\ell}{\delta}} \left|F(t)\right|^2\dt +  \dfrac{(1+\eta)^2\ell^2x^2}{\eta\pi} \int_{\frac{\ell}{\delta}}^\infty\dfrac{|F(t)|^2}{t^2}\dt.  
					\end{split} 
				\end{align}

				\subsection{The error term} 
				In order to complete the proof of Theorem \ref{thetaMeanValue}, it remains to verify that all terms on the right-hand side of \eqref{1_57am_30_04}, beginning with the second, are suitably small. It is immediate that \(\frac{\delta^2x^2}{2}\) is negligible. Our next task is to estimate the integrals related to \(F(t)\). 
				From \eqref{11_37pm}, \eqref{7_10_29_04} and (1) in Lemma \ref{Lemmaprime1}, for any $a\geq b\geq0$ and $\eta_1>0$:
				\begin{align} \label{11_26pm}
					\begin{split} 
						\int_{a}^b|F(t)|^2\,\dt  &   \leq \left(1+\dfrac{1}{\eta_1}\right)\int_{a}^b\left|\frac{\zeta'}{\zeta}(1+2it)+\frac{1}{2it}\right|^2\!\!\dt+ (1+\eta_1)\int_{a}^b\left|\sum_p \frac{\log p}{p^{\frac{1}{2}+it}(p^{1+2it}-1)}\right|^2\!\!\dt\\
						& \leq \dfrac{1}{2}\left(1+\dfrac{1}{\eta_1}\right)\int_{2a}^{2b}\left|\frac{\zeta'}{\zeta}(1+it)+\frac{1}{it}\right|^2\!\!\dt+2.9636(1+\eta_1)(b-a).
					\end{split} 
				\end{align} 
				%\footnote{One can prove that asymptotically we have the following expression:
					%	$
					%	\displaystyle\int_{0}^T\left|\frac{\zeta'}{\zeta}(1+2it)+\frac{1}{2it}\right|^2\dt \sim 	\displaystyle\sum_{n=2}^\infty\dfrac{\Lambda^2(n)}{n^2}{T}.
					%	$ 
					\subsubsection{The first integral in \eqref{1_57am_30_04}} Here we split as follows
					\begin{align} \label{12_56am}
						\int_{0}^{\frac{\ell}{\delta}} \left|F(t)\right|^2\dt = \int_{0}^{\frac{1}{4}} \left|F(t)\right|^2\dt + \int_{\frac{1}{4}}^{\frac{10^4}{2}} \left|F(t)\right|^2\dt + \int_{\frac{10^4}{2}}^{\frac{\ell}{\delta}} \left|F(t)\right|^2\dt:= I_1 + I_2 + I_3. 
					\end{align}
					To bound $I_1$, we apply \eqref{11_26pm} with $a=0$, $b=\frac{1}{4}$, Lemma \ref{11_18pm} and $\eta_1=1.5307$ obtaining $I_1\leq4.8$. To bound $I_2$, we use computational methods\footnote{More specifically, we have checked this numerically in python with mpmath.quad and mpmath.zeta from the mpmath-package.} to get
					\begin{align*}
						\int_{\frac{1}{2}}^{10^4}\left|\frac{\zeta'}{\zeta}(1+it)+\frac{1}{it}\right|^2\dt  \leq  8400.
					\end{align*}
					Then, applying \eqref{11_26pm} with $a=\frac{1}{4}$, $b=\frac{10^4}{2}$ and $\eta_1=0.5324$ we obtain $I_2\leq 34794.8$. Finally, to bound $I_3$, we start applying \eqref{7_10_29_04} with $\eta_2=10^{-8}$ and Lemma \ref{secondmoment} with $T=\frac{2\ell}{\delta}\geq 4\cdot 10^{13}$ to see that
					\begin{align*}
						\int_{10^4}^{\frac{2\ell}{\delta}}\left|\frac{\zeta'}{\zeta}(1+it)+\frac{1}{it}\right|^2\dt \leq (1+10^{-8})\int_{10^4}^{\frac{2\ell}{\delta}}\left|\frac{\zeta'}{\zeta}(1+it)\right|^2\!\!\dt + (1+10^{8})\int_{10^4}^{\frac{2\ell}{\delta}}\,\dfrac{1}{t^2}\dt  \leq 1.6113\,\dfrac{\ell}{\delta},
					\end{align*}
					where we used that $\ell>2$. Thus, applying \eqref{11_26pm} with $a=\frac{10^4}{2}$, $b=\frac{\ell}{\delta}$ and $\eta_1=0.5213$ {we obtain} $I_3\leq 6.8597\frac{\ell}{\delta} - 22542.6$. Finally, in \eqref{12_56am} we get
					\begin{align}  \label{1_27am}
						\int_{0}^{\frac{\ell}{\delta}} \left|F(t)\right|^2\dt\leq 6.8598\,\dfrac{\ell}{\delta}.
					\end{align} 
					\subsubsection{The second integral in \eqref{1_57am_30_04}} For $t\geq \frac{\ell}{\delta}>2\cdot 10^{13}$, we apply \eqref{7_10_29_04} with $\eta_2=10^{-8}$ and Lemma \ref{secondmoment} with $T=2t$ to get
					\begin{align}  \label{1_20AM}
						\int_{\frac{2\ell}{\delta}}^{2t}\left|\frac{\zeta'}{\zeta}(1+iu)+\frac{1}{iu}\right|^2\du\leq (1+10^{-8})	\int_{\frac{2\ell}{\delta}}^{2t}\left|\frac{\zeta'}{\zeta}(1+iu)\right|^2\!\!\du + (1+10^{8})	\int_{\frac{2\ell}{\delta}}^{2t}\dfrac{1}{u^2}\du  \leq 1.6113\,t.
					\end{align}
					Now, since $F(t)=O(\log t)$ as $t\to\infty$, we can use integration by parts to arrive at
					\begin{align*}
						\int_{\frac{\ell}{\delta}}^\infty\dfrac{|F(t)|^2}{t^2}\dt = 2\int_{\frac{\ell}{\delta}}^\infty\left(\int_{\frac{\ell}{\delta}}^t|F(u)|^2\du\right)\dfrac{1}{t^3}\dt.
					\end{align*}
					\added{Now,}\deleted{Then, } in \eqref{11_26pm} \deleted{we }choose \deleted{with} $a=\frac{\ell}{\delta}$, $b=t$, $\eta_1=0.7373$, and \replaced{use}{using} \eqref{1_20AM} \added{, to} \deleted{we} see that
					\begin{align} \label{1_26am2}
						\begin{split} 
							\int_{\frac{\ell}{\delta}}^\infty\dfrac{|F(t)|^2}{t^2}\dt & \leq \left(1+\dfrac{1}{\eta_1}\right)\int_{\frac{\ell}{\delta}}^\infty\left(\int_{\frac{2\ell}{\delta}}^{2t}\left|\frac{\zeta'}{\zeta}(1+iu)+\frac{1}{iu}\right|^2\du\right)\dfrac{1}{t^3}\dt+5.9272(1+\eta_1)\int_{\frac{\ell}{\delta}}^\infty\left(t-\dfrac{\ell}{\delta}\right)\dfrac{1}{t^3}\dt \\
							& \leq 8.9454\,\dfrac{\delta}{\ell}.
						\end{split}
					\end{align} 
					
					Finally, \replaced{inserting}{we insert} \eqref{1_27am} and \eqref{1_26am2} in \eqref{1_57am_30_04}, \deleted{and} using that $\ell<2.0001$ and $\delta\leq 10^{-13}$,
					\begin{align*}
						\!\!J_\theta(x,\delta)
						& \leq   \left(1+{\eta}\right)\cdot2.2258\cdot\delta\,\log^2\left(\dfrac{2.0001}{\delta}\right)x^2 + \left(1+\dfrac{1}{\eta}\right)^2\dfrac{\delta^2x^2}{2} + \left(\dfrac{15.8052(1+\eta)^2\ell}{\eta\pi}\right)\delta x^2 \\
						& < \left[\left(1+{\eta}\right)\cdot2.2258 + \dfrac{31.612(1+\eta)^2}{\eta\pi\log^2\left({2.0001}\cdot{10^{13}}\right)} \right]\delta\,\log^2\left(\dfrac{2.0001}{\delta}\right)x^2 + \left(1+\dfrac{1}{\eta}\right)^2\dfrac{\delta^2x^2}{2}.
					\end{align*} 	
					\deleted{We optimize t}The expression in \added{the} brackets \added{is optimized by} choosing $\eta=0.0693$ \deleted{to obtain the desired result}.

					\section{Proof of Theorem \ref{6_35pm2}} \label{10_39pm2}
					We closely follow the argument in \cite[Section 4]{Dudek}, some of which is restated here for convenience of the reader.
					
					Let $a\in [10^{-13},1)$, and $x \geq  10^{13}$, Theorem \ref{thetaMeanValue} implies
					\begin{equation}\label{mainIneqWithIntervalaxx}
						\int_{ax}^x \big(\theta((1+\delta)y)-\theta(y)-\delta y\big)^2\dy 
						\leq  2.5571\cdot\delta\,\log^2\left(\dfrac{2.0001}{\delta}\right)x^2
					\end{equation}
					holds for any $\delta \in (0,10^{-13}]$. Assume there is no Goldbach number in the interval $(x,x+h]$ for any $1\leq h \leq x$ Following \emph{verbatim} \cite[Section 4]{Dudek}, this implies that $$\int_{ax}^x \big(\theta((1+\delta)y)-\theta(y)-\delta y\big)^2\dy>\frac{\delta^2 x^3}{3}\left(\frac{1}{8}-a^3\right),$$ under the assumption $\delta \leq \frac{h}{2x}$. By (\ref{mainIneqWithIntervalaxx}), we \added{then} have
					\begin{equation*}
						\frac{\delta^2 x^3}{3}\left(\frac18 - a^3\right) < 2.5571\,\delta \log^2\left(\frac{2.0001}{\delta}\right)x^2,
					\end{equation*} 
					\replaced{Choosing}{We choose} $\delta=h/(2x)$ with $h=C(\log x)^2$, with $C>1$\deleted{. Then}, we have
					\begin{align*}
						\frac{C(\log x)^2}{6}\left(\frac18 - a^3\right) &< 2.5571\log^2\left(\frac{2.0001\cdot 2x}{C(\log x)^2}\right) \\ 
						&= 2.5571(\log x)^2\left(1+\frac{1}{\log x}\log\left( \frac{4.0002}{C(\log x)^2}\right)\right)^2 \leq2.5571(\log x)^2. 
					\end{align*}
					\replaced{This inequality is contradicted}{We contradict this inequality} when $$C>\frac{6\cdot 2.5571}{\frac18 - a^3}.$$ \replaced{Choosing $a=10^{-13}$, this}{Minimizing this we find the smallest value at  This} implies we can take $C=122.75$, as long as ${122.75(\log x)^2}/{2x}\leq 10^{-13}$, which is true whenever $x\geq 1.1\cdot 10^{18}$. In \cite{Herzog}, Goldbach's conjecture is proven up to $4\cdot 10^{18}$, so this finishes the proof.

					\appendix
					\setcounter{section}{0}             % reinicia numeración
					\renewcommand{\thesection}{\Alph{section}}
					\section{Some sums over primes} \label{6_46AM}

					\begin{lemma} \label{Lemmaprime1} We have the following bounds:
						\begin{enumerate}
							\item $$ \sum_p \frac{\log p}{p^{\frac{1}{2}}(p-1)}<1.7215,$$
							\item $$ \sum_{n=1}^\infty\dfrac{\Lambda^2(n)}{n^2}<0.8053,$$
							\item $$\sum_{p}\dfrac{\log^2p}{p^2-p}<0.982.$$
						\end{enumerate}
					\end{lemma}
					\begin{proof}
						To prove (1), we use the fact that $p_n>n\log n$ for all $n\geq1$, where $p_n$ is the n-th prime number (see \cite[Corollary, p. 69]{RSch}). Thus, letting $n_0=26355867$,
						\begin{align*}
							\sum_p \frac{\log p}{p^{\frac{1}{2}}(p-1)} & =\sum_{p\leq p_{n_0}} \frac{\log p}{p^{\frac{1}{2}}(p-1)} + \sum_{p> p_{n_0}} \frac{\log p}{p^{\frac{1}{2}}(p-1)} \\
							& < \sum_{p\leq p_{n_0}} \frac{\log p}{p^{\frac{1}{2}}(p-1)} + \sum_{n> {n_0}} \frac{\log (n\log n)}{(n\log n)^{\frac{1}{2}}(n\log n-1)} \\
							& < \sum_{p\leq p_{n_0}} \frac{\log p}{p^{\frac{1}{2}}(p-1)} + \int_{{n_0}}^\infty \frac{\log (x\log x)}{(x\log x)^{\frac{1}{2}}(x\log x-1)}\dx < 1.721381 + 0.000104 < 1.7215,
						\end{align*} 
						where the numerical bounds are evaluated computationally. With a similar approach we get (3). To prove (2), note that
						$$\sum_{n=1}^\infty \frac{\Lambda(n)^2}{n^2} = \sum_p \log^2 p \sum_{k=1}^\infty \frac{1}{p^{2k}} = \sum_{p} \log^2 p\left(\frac{1}{1-p^{-2}}-1\right) = \sum_p \frac{\log^2 p}{p^2-1}.$$
						Then, we bound this sum as in (1). 
					\end{proof}
					
					\begin{lemma} \label{Lemmaprime2} Assume RH. Then, for all $x\geq 10^{13}$ we have
						$$
						\sum_{n \leq x}\dfrac{\Lambda^2(n)}{n}\leq \dfrac{\log^2x}{2} + 4.5222.
						$$
						
					\end{lemma} 
					\begin{proof}
						Using (3) from Lemma \ref{Lemmaprime1},
						\begin{align}  \label{1_40am}
							\begin{split} 
								\sum_{n\leq x} \frac{\Lambda^2(n)}{n}& \leq \sum_{p\leq x} \frac{\log^2p}{p} + \sum_{p\leq \sqrt{x}}\log^2p\sum_{k=2}^\infty\dfrac{1}{p^k}= \sum_{p\leq x} \frac{\log^2p}{p}  + \sum_{p}\dfrac{\log^2p}{p^2-p}< \sum_{p\leq x} \frac{\log^2p}{p} + 0.982.
							\end{split}
						\end{align} 
						To bound the sum on the right hand-side of \eqref{1_40am} we use integration by parts and the bound $\theta(y)<y+\sqrt{y}\log^2y/8\pi$, for $y>0$ (by \cite[Theorem 10, Eq (6.5)]{Sch}), where $\theta(y)=\sum_{p\leq y}\log p$. Thus
						\begin{align*}
							\sum_{p\leq x} \frac{\log^2p}{p}  & =\frac{\log^22}{2}+\dfrac{\log x}{x}\theta(x)-\dfrac{\log 3}{3}\theta(3^{-})+\int_{3}^x\left(-\dfrac{\log y}{y}\right)'\theta(y)\dy \\
							& < \frac{\log^22}{2}-\dfrac{\log3\log 2}{3}+\log 3+\int_{3}^x\dfrac{\log y}{y}\dy+\dfrac{\log^3x}{8\pi\sqrt{x}}+\int_{3}^x\left(-\dfrac{\log y}{y}\right)'\dfrac{\sqrt{y}\log^2y}{8\pi}\dy\\
							&< \frac{\log^22}{2}-\dfrac{\log3\log 2}{3}+\log 3+\dfrac{\log^2x}{2}-\dfrac{\log^23}{2}+\dfrac{\log^3x}{8\pi\sqrt{x}}+\int_{3}^{\infty}\left(-\dfrac{\log y}{y}\right)'\dfrac{\sqrt{y}\log^2y}{8\pi}\dy \\
							& < \dfrac{\log^2x}{2} +  3.5401 +  \dfrac{\log^3x}{8\pi\sqrt{x}} < \dfrac{\log^2x}{2} +  3.5402,
						\end{align*}
						where in the final inequality we used that $x\geq 10^{13}$. Inserting this in \eqref{1_40am} we arrive at the desired result.
					\end{proof}

				\end{document}